\documentclass[11pt]{article}

\usepackage[dvipdfmx]{graphicx}
\usepackage{amssymb}
\usepackage{unit}
\usepackage{color}
\usepackage{mathrsfs}
\usepackage{amsmath}
\usepackage{subfigure}
\usepackage{color}
\usepackage{amsthm}
\usepackage{amscd}

\theoremstyle{definition}
\newtheorem{theorem}{Theorem}
\newtheorem{definition}[theorem]{Definition}
\newtheorem{lemma}[theorem]{Lemma}

\newtheorem{corollary}[theorem]{Corollary}

\newcommand{\itO}{\mathit{\Omega}}
\newcommand{\mO}{\mathcal{O}}
\newcommand{\itS}{\mathit{\Sigma}}
\newcommand{\bbT}{\mathbb{T}}
\newcommand{\bbR}{\mathbb{R}}
\newcommand{\cl}{\mathrm{cl}}
\newcommand{\thz}{{\theta_0}}

\begin{document}
\title{%
Invariant Sets in Quasiperiodically Forced Dynamical Systems\thanks{The work was partially supported by JSPS Postdoctoral Fellowships for Research Abroad, JST-CREST Program No. JP-MJCR15K3, and JSPS-KAKEN No. 15H03964 (Y.S.), and by Office of Naval Research Grant No. N00014-14-1-0633 and the ARO-MURI grant W911NF-14-1-0359 (I.M.)}
}%
\author{%
Yoshihiko Susuki\footnote{Yoshihiko Susuki is with Department of Electrical and Information Systems, Osaka Prefecture University, Japan, \texttt{susuki@eis.osakafu-u.ac.jp}, \texttt{susuki@ieee.org}}
~~~
Igor Mezi\'c\footnote{Igor Mezi\'c is with Department of Mechanical Engineering, University of California, Santa Barbara, United States. \texttt{mezic@engineering.uscb.edu}}
}%
\date{November 2019}

\maketitle

\begin{abstract}
This paper addresses structures of state space in quasiperiodically forced dynamical systems.  
We develop a theory of ergodic partition of state space in a class of measure-preserving and dissipative flows, which is a natural extension of the existing theory for measure-preserving maps.  
The ergodic partition result is based on eigenspace at eigenvalue 0 of the associated Koopman operator, which is realized via time-averages of observables, and provides a constructive way to visualize a low-dimensional slice through a high-dimensional invariant set.  
We apply the result to the systems with a finite number of attractors and show that the time-average of a continuous observable is well-defined and reveals the invariant sets, namely, a finite number of basins of attraction.  
We provide a characterization of invariant sets in the quasiperiodically forced systems.  
A theoretical result on uniform boundedness of the invariant sets is presented. 
The series of theoretical results enables numerical analysis of invariant sets in the quasiperiodically forced systems based on the ergodic partition and time-averages.  
Using this, we analyze a nonlinear model of complex power grids that represents the short-term swing instability, named the coherent swing instability.  
We show that our theoretical results can be used to understand stability regions in such complex systems.  
\end{abstract}
 
%


\section{Introduction}
\label{intro}

Methods of ergodic theory \cite{Arnold:1968,Eckmann_RMP57} have been applied for solving a number of problems in science and technology such as control of mixing of fluid flows \cite{DAlessandro_IEEETAC44}, 
analysis of traffic jams \cite{Gray_JSP105} and of quantum many-body models \cite{Lee_PRL87}, performance evaluation of ecosystem models \cite{Pietsch_TP25}.  
A visualization method for invariant sets of discrete-time dynamical systems (or maps) based on ergodic theory was developed in \cite{Mezic_CHAOS9}.  
This method is based on the \emph{ergodic partition} or ergodic decomposition \cite{Peterson:1983} of state space in a given dynamical system, which is associated with the eigenspace at eigenvalue $1$ of the Koopman operator \cite{Marko_CHAOS22}.  
This method partitions a (compact and metric) state space $X$ using joint level sets of time-averages of a basis of functions defined on $X$.   
This leads to an approximation of invariant sets on which the dynamics of the system are ergodic.  
In an ergodic invariant set, almost all points are accessible in the sense that the initial conditions in this set thoroughly sample the set.  
The method enables the visualization of low-dimensional (in practice usually two-dimensional) slices through high-dimensional invariant structures and offers a computationally tractable method for analyzing global structures of state space in discrete-time nonlinear models.  
The numerical methods associated with the ergodic partition theory have been developed: see \cite{Marko_CHAOS22} and references therein.  

The purpose of this paper is to extend the ergodic partition theory and exploit it to provide global analysis of state space in \emph{quasiperiodically forced} dynamical systems.  
Quasiperiodicity is one of the three types of commonly observed dynamics in both deterministic models and experiments: see, e.g., \cite{Das_EPL114,Glaz_JFM828}.   
Quasiperiodically forced systems are important objects in nonlinear dynamics and have been studied by many groups of researchers:  phenomenology of chaotic attractors \cite{Moon_PL111A}, analysis of dynamics described by the Schr\"odinger equations \cite{Bondeson_PRL55} and of quantum chaos \cite{Pomeau_PRL56}, and dynamical systems analysis with applications to fluid mixing \cite{Wiggins_PLA124,Wiggins:1992}. 
In the present paper, we study the global structure of state space in quasiperiodically forced systems from the viewpoint of ergodic partition.  
First, we develop a theory of ergodic partition of state space in measure-preserving and dissipative, continuous-time dynamical systems or \emph{flows}.   
This is a natural extension of the existing theory for measure-preserving maps in \cite{Mezic_CHAOS9}. 
The ergodic partition is again based on eigenspace at eigenvalue 0 of the associated Koopman group. 
Related analyses of time-periodic flows via spectral properties of linear operators are reported in \cite{Froyland_Nonlinear30,Giannakis_Preprint:2017}. 
Since an arbitrary quasiperiodically forced system with a smooth invariant measure is transformed into an autonomous system determining a measure-preserving flow (see Section~\ref{sec:set-up}), the ergodic partition theory is used for visualization of invariant sets for the quasiperiodically forced system.  
Next, we provide a characterization of invariant sets in quasiperiodically forced systems.  
In general, understanding invariant structures in the quasiperiodically forced systems is not easy because the associated portraits of state space change with time in an aperiodic manner.  
Here, we introduce a notion of \emph{uniform boundedness} of invariant sets.  
A theoretical result on uniform boundedness of invariant sets is presented in this paper: see Theorem~\ref{ther:qpsets-IM}, and Corollaries~\ref{thm:2-IM} and \ref{thm:NEW}.   
This clarifies a dynamical feature of invariant sets in the quasiperiodically forced systems that is directly determined by the ergodic partition and its associated visualization technique.  
Also, we apply the theory developed above to analysis of a nonlinear model of complex power grids that represents the short-term swing instability, which we name the Coherent Swing Instability 
(CSI) \cite{Susuki_JNLS09}. 
Preliminary results in this paper were published in the conference proceedings \cite{Susuki_CDC09,Mezic_IFACNOLCOS16}. 
This paper contains detailed discussion of the theory with a new proof and a set of new numerical simulations. 

The rest of this paper is organized as follows.
In Section~\ref{sec:set-up} we introduce set-up and mathematical preliminaries from dynamical systems theory.  
A theory of ergodic partition in measure-preserving flows is developed in Section~\ref{sec:ergodic}.  
Based on the result, the basin of attraction in the dissipative case is explored in \eqref{eq:basingen}.   
In Section~\ref{sec:chara} we provide a new characterization of invariant sets in the quasiperiodically forced dynamical systems.  
The developed theory is applied in Section~\ref{sec:example} to two simple examples of the quasiperiodically forced system and to a nonlinear model of the CSI phenomenon of a power grid.  
Conclusions of this paper are presented in Section~\ref{sec:outro}.

\section{Set-Up and Definitions}
\label{sec:set-up}

\subsection{Quasiperiodically Forced Dynamical Systems}

Throughout this paper, we address the following quasiperiodically forced dynamical system that evolves on a finite-dimensional metric space $M$:  for $\vct{m}\in{M}$ and $t\in\mathbb{R}$, 
\begin{equation}
\frac{{d}\vct{m}}{{d} t}=\vct{g}(\vct{m},t). 
\label{eqn:org_syst}
\end{equation}
The function $g$ is assumed to be \emph{quasiperiodic} on $t$ in the sense of Moser \cite{Moser_SIAMRev8}, that is, 
\[
\vct{g}(\vct{m},t)=\vct{G}(\vct{m},\itO_1t,\itO_2t,\ldots,\itO_Nt),
\]
where $\vct{G}(\vct{m},\theta_1,\theta_2,\ldots,\theta_N)$ is assumed to be a smooth vector-valued function and of period $2\pi$ in $\theta_1,\theta_2,\ldots,\theta_N$.  
The real numbers $\itO_1,\itO_2,\ldots,\itO_N$ are the $N$ basic angular frequencies and assumed to be rationally independent: there exists no $N$-dimensional integer vector $(k_1,k_2,\ldots,k_N)^\top$ ($\top$ stands for the transpose operation of real-valued vectors) in which the entries do not all vanish, satisfying the resonance relation:
\begin{equation}
\sum^{N}_{i=1}k_i\itO_i=0.
\label{eqn:incommen}
\end{equation}
The system \eqref{eqn:org_syst} is non-autonomous and transformed into an autonomous system defined on the augmented state space $X:=M\times\bbT^N$, in the same manner as in \cite{Wiggins:1992}.    
By introducing the $N$ variables $\theta_i=\itO_it\in\bbT$ ($i=1,2,\ldots,N$), we have
\begin{equation}
\frac{d\vct{m}}{dt}=\vct{G}(\vct{m},\theta_1,\theta_2,\ldots,\theta_N), \qquad
\frac{d\theta_i}{dt}=\itO_i \qquad i=1,2,\ldots,N.
\label{eqn:aug_syst}
\end{equation}
In this way, by taking trajectories of the augmented system \eqref{eqn:aug_syst} on the augmented state variable $\vct{x}:=(\vct{m},\theta_1,\theta_2,\ldots,\theta_N)\in X$, a \emph{flow}, that is, one-parameter group of diffeomorphisms is defined for the quasiperiodically forced system \eqref{eqn:org_syst}.

\subsection{Measure-Preserving Flows and Partition of State Space}
\label{subsec:mpf}

In Section~\ref{sec:ergodic}, as one type of flows induced by \eqref{eqn:aug_syst}, we will investigate a measure-preserving flow defined on a class of probability spaces, where ergodic theory has been developed \cite{Peterson:1983,Mane:1987}.  
The probability space considered here corresponds to the tuple $({X},\mathfrak{B}_{X},\mu)$, where ${X}$ is a compact metric space, $\mathfrak{B}_{X}$ is the Borel $\sigma$-algebra of ${X}$, and $\mu$ is a probability measure.  
A measure-preserving flow of the probability space is a one-parameter group of measure-preserving diffeomorphisms $\vct{S}^t: {X}\to{X}$, $t\in\bbR$ such that $\vct{S}^t$ is measure-preserving (for all $A\in\mathfrak{B}_{X}$ we have $\mu(\vct{S}^{-t}(A))=\mu(A)$), $\vct{S}^0$ the identity, and $\vct{S}^{t_1+t_2}=\vct{S}^{t_1}\circ\vct{S}^{t_2}$ for $t_1,t_2\in\bbR$.  

The important notion which we utilize throughout this paper is the partition of state space in dynamical systems.  
One motivation behind the following definitions is that we aim to locate such a partition that plays a role in the data-driven estimation of certain statistical properties of non-ergodic measure-preserving dynamical systems; we will show this later as the ergodic partition. 
The definitions of partition and measurable partition from \cite{Mezic_CHAOS9} are the following. 
\begin{definition}
A partition of ${X}$ is a family $\zeta$ of sets satisfying
\[
A,B\in\zeta \quad \Rightarrow \quad \mu(A\cap B)=0, \quad  
\mu\left({X}\setminus\bigcup_{A\in\zeta}A\right)=0.
\] 
The element of $\zeta$ containing a point $\vct{x}\in{X}$ is denoted by $\zeta(\vct{x})$.  
A partition $\zeta$ of ${X}$ is said to be measurable if there exists a countable family $\mathfrak{D}$ of measurable sets $\{D_i\}$ such that every $D_i$ is a union of elements of $\zeta$, and for any pair $A_1,A_2$ of elements of $\zeta$ there exists $D_j\in\mathfrak{D}$ such that $A_1\subset D_j$ and $A_2\subset D_j^{\rm c}$, where $D_j^{\rm c}$ stands for the complement of $D_j$ in ${X}$. 
The family $\mathfrak{D}$ is called a basis for $\zeta$.
\end{definition}

In Section~\ref{sec:ergodic} we use a product operation on the set of partitions of ${X}$.  
The product operation is defined in \cite{Mane:1987} as follows.
\begin{definition}
If $\zeta_n$, $n=1,\ldots,N$ are partitions, we then define their product $\vee^{N}_{n=1}\zeta_n$ as the partition whose elements are the sets of the form $\cap^{N}_{i=1}{A_i}$, 
for $A_i\in\zeta_i$, $i=1,\ldots,N$, satisfying $\mu\left(\cap^{N}_{i=1}A_i\right)\neq 0$. 
For a countable sequence of partitions, the notation $\vee^{\infty}_{n=1}\zeta_n$ will be used for the $\sigma$-algebra generated by $\cup^\infty_{n=1}\zeta_n$.   
\end{definition}

\subsection{Koopman Group}

Consider a flow $\vct{S}^t: {X}\to{X}$ ($t\in\mathbb{R}$) on a finite-dimensional space ${X}$ 
and denote by $\mathcal{F}$ a space of scalar-valued functions defined on ${X}$: $\mathcal{F}\ni f: {X}\to\mathbb{C}$, which we call the space of \emph{observables}.  
In the following, we suppose that the existence and uniqueness of solutions associated with the flow hold for all $t$.  
Then, we define the one-parameter group of linear operators $U^t$ ($t\in\bbR$) 
as a group of composition operators with $\vct{S}^t$: for $f\in\mathcal{F}$, 
\[
(U^tf)(\vct{x}):=f(\vct{S}^t(\vct{x}))=(f\circ\vct{S}^t)(\vct{x}).
\]
where it is known as the \emph{Koopman group} \cite{Lasota:1994,Marko_CHAOS22}.  
For spaces as $\mathcal{F}=\mathcal{C}^0(X)$ and $\mathcal{L}^p(X)$, the Koopman group is strongly continuous.  
Although the original flow $\vct{S}^t$ is possibly described by a \emph{nonlinear} differential equation 
and evolves on the \emph{finite}-dimensional space ${X}$, 
the operators $U^t$ are \emph{linear} but evolve on the \emph{infinite}-dimensional space $\mathcal{F}$.   
The \emph{eigenvalue} $\lambda\in\mathbb{C}$ and \emph{eigenfunction} $\phi_\lambda\in\mathcal{F}\setminus\{0\}$ of the Koopman group are defined as follows:
\[
(U^t\phi_\lambda)(\vct{x})=\exp(\lambda t)\phi_\lambda(\vct{x}).
\] 
The notion of the Koopman group and its spectral characterization will be used in the following sections.

\section{Ergodic Partition in Measure-Preserving Flows}
\label{sec:ergodic}

For completeness, in this section we extend the theory of ergodic partition in \cite{Mezic_CHAOS9} to the measure-preserving flow $\vct{S}^t$ on the probability space $(X,\mathfrak{B}_{X},\mu)$, which is introduced in Section~\ref{subsec:mpf}.  
This section consists of a few lemmas and a theorem.  
Their proofs are almost identical to the proofs for maps \cite{Mezic_CHAOS9} and appear in Appendices~\ref{appsec:1} and \ref{appsec:2}.  
We denote by $\mathcal{L}^1_\mu({X})$ the space of all $\mu$-integrable functions on ${X}$ and by $\mathcal{C}({X})$ the space of all real-valued continuous functions on ${X}$ endowed with the sup norm. 

Roughly speaking, an ergodic partition is a partition of the state space ${X}$ into invariant sets on which the dynamics are ergodic.  
The precise definition of ergodic partition is the following. 
\begin{definition}
A measurable partition $\zeta$ of ${X}$ is said to be \emph{ergodic} under the flow $\vct{S}^t$ if for any element $A$ of $\zeta$, (i) $A$ is invariant under $\vct{S}^t$, and (ii) there exists an invariant probability measure $\mu_A$ on $A$ such that the restriction of $\vct{S}^t$ to $A$, denoted by $\vct{S}^t|_A$, is an ergodic diffeomorphism on $A$, and for all $f\in\mathcal{L}^1_\mu({X})$, 
\begin{equation}
\int_{X} f{d}\mu=\int_{X}\left[\int_{A=:\zeta(\vct{x})} f|_A{d}\mu_A\right]{d}\mu(\vct{x}), 
\label{eqn:edt}
\end{equation}
where $f|_A$ stands for the restriction of $f$ to the ergodic element $A$, and $\mu(\vct{x})$ is again a probability measure on $X$.  
\end{definition}

It is remarked that the ergodic partition is defined for the entire state space $X$ with $\sigma$-algebra, but it can be relaxed in context of the $\sigma$-algebra of invariant sets which can be parts of the state space. 

Here, we denote by $f^\ast$ the time-average of $f$ under the flow $\vct{S}^t$ if the right-hand side of
\begin{equation}
f^\ast(\vct{x}):=\lim_{T\rightarrow\infty}\frac{1}{T}\int^T_0 f(\vct{S}^t(\vct{x})){d} t,
\end{equation}
exists for almost every (a.e.) point $\vct{x}\in{X}$ with respect to the measure $\mu$.  
Birkhoff's Ergodic Theorem \cite{NS,Arnold:1968} shows that for all $f\in\mathcal{L}^1_\mu({X})$, (i) $f^\ast(\vct{x})$ exists for a.e. point $\vct{x}\in{X}$ (with respect to $\mu$); (ii) $f^\ast(\vct{S}^t(\vct{x}))=f^\ast(\vct{x})$ for a.e. point $\vct{x}\in{X}$; and (iii) $\int_{X}f{d}\mu=\int_{X}f^\ast{d}\mu$.  
Here, we let $\itS$ be the set of points $\vct{x}\in{X}$ such that $f^\ast(\vct{x})$ exists for \emph{all} $f\in\mathcal{C}({X})$, and $\itS(f)$ the set of points $\vct{x}\in{X}$ such that $f^\ast(\vct{x})$ exists for a \emph{particular} $f\in\mathcal{C}({X})$.  
The following lemma is standard (see page 129 of \cite{Mane:1987} for discrete-time dynamical systems (maps): it is obvious for flows).    
\begin{lemma}
For a countable and dense set $S$ in $\mathcal{C}({X})$,
\begin{equation}
\itS=\bigcap_{f\in S}\itS(f).
\end{equation}
\end{lemma}

Now, we have a set $\itS$ such that the time-averages of all continuous functions on ${X}$ exist.    
Then, the complimentary set $\itS^{\rm c}$ is of measure zero.  
This is because according to Birkhoff's Ergodic Theorem (i), each $\{\itS(f)\}^{\rm c}$ is of measure zero, and thus $\itS^{\rm c}:=\cup_{f\in S}\{\itS(f)\}^{\rm c}$ is a countable union of sets with zero measure, which is again of measure zero.  
The next lemma shows that the time-average of a continuous function induces a measurable partition on ${X}$.  
\begin{lemma}
\label{lemma:1}
Let $f$ be a continuous function on ${X}$.  
The family of level sets of $f^\ast$, 
\[
A_\alpha:=\{\vct{x} : \vct{x}\in\itS, f^\ast(\vct{x})=\alpha\},~~~\alpha\in\bbR,
\]
is a measurable partition of ${X}$.  
We denote by $\zeta_f$ this partition and call it the partition induced by the function $f$.  
\end{lemma}
\begin{proof}
See Appedix~\ref{appsec:1}.
\end{proof}

Now, we can prove the first theorem stating that $\zeta_f$ induces the ergodic partition of ${X}$, which holds for general measure-preserving flows including the augmented system \eqref{eqn:aug_syst} for the quasiperiodically forced system \eqref{eqn:org_syst}. 
\begin{theorem}
\label{thm:1}
Consider the measure-preserving flow $\vct{S}^t: X\to X$ ($t\in\mathbb{R}$) associated with
\[
\left.\frac{d\vct{S}^t(\vct{x})}{dt}\right|_{t=0} =\vct{F}(\vct{x}) \quad \textrm{for each}~\vct{x}\in X
\]
where $\vct{F}: X\to\mathrm{T}X$ (tangent bundle of $X$) is a nonlinear vector field.   
Let $\zeta\sub{e}$ be the product of measurable partitions of ${X}$ induced by every $f\in S$:
\[
\zeta\sub{e}=\bigvee_{f\in S}\zeta_f.
\]
Then, $\zeta\sub{e}$ is the ergodic partition of ${X}$. 
\end{theorem}
\begin{proof}
See Appendix~\ref{appsec:2}.
\end{proof}

By combining the above theorem and lemma~20 in \cite{Mezic_PD197}, we have the following corollary based on a finite number of basis functions. 
\begin{corollary}
Assume there exists a complete system of functions $\{f_i\}$, $f_i\in\mathcal{C}(X)$, $i\in\mathbb{N}_{>0}$ (set of all natural numbers except for $0$) i.e. finite linear combinations of $f_i$ are dense in $\mathcal{C}(X)$.  
The ergodic partition $\zeta\sub{e}$ is
\[
\zeta\sub{e}=\bigvee_{i\in\mathbb{N}_{>0}}\zeta_{f_i}.
\]
\end{corollary}

For a given function $f$, it is obvious that each level set $A_\alpha=\{\vct{x} : \vct{x}\in\itS, f^\ast(\vct{x})=\alpha\}$ ($\alpha\in\bbR$) as an element of $\zeta_f$ is invariant under $\vct{S}^t$.  
In \cite{Mezic_CHAOS9} the authors proposed to use the level sets $A_\alpha$, which are directly computed with the time-averaging of $f$ under $\vct{S}^t$, for identifying and visualizing invariant sets in discrete-time dynamical systems possessing a smooth invariant measure.  
Theorem~\ref{thm:1} implies that the same computation can be used for visualization of invariant sets for the measure-preserving flow, namely, continuous-time dynamical systems with a smooth invariant measure.  
Since, as shown in Section~\ref{sec:set-up}, the quasiperiodically forced system \eqref{eqn:org_syst} defined on the state space $M$ is transformed to the flow defined on the augmented state space $X=M\times\bbT^N$, the ergodic partition theory is applicable to visualization of invariant sets in $X$ for the original quasiperiodically forced system \eqref{eqn:org_syst} possessing a smooth invariant measure.  
However, in the quasiperiodically forced system \eqref{eqn:org_syst} we are interested in dynamics on $M$ not on $X=M\times\mathbb{T}^N$.  
While for $N=1$ it is clear that every intersection of the invariant set in $M\times\mathbb{T}^1$ with $M$ is an invariant set of the associated Poincar\'e map.  
The relationship in the quasiperiodically forced system (namely, $N\geq 2$) is much less clear. 
We make it precise in Section~\ref{sec:chara}.

Here, let us introduce the connection of time-average $f^\ast$ and the Koopman group.  
It is obvious that for any $f\in\mathcal{L}^1_\mu(X)$ its time-average $f^\ast$ is an eigenfuction of the operators $U^t$ at eigenvalue $0$:  for a.e. point $\vct{x}\in{X}$ with respect to $\mu$,
\[
(U^tf^\ast)(\vct{x})
=f^\ast(\vct{S}^t(\vct{x}))
=f^\ast(\vct{x})
=\exp(0t)f^\ast(\vct{x}). 
\]
Since the eigenfunctions at $0$ are invariant under the flow, we have a complete characterization of (possibly non-smooth) invariant sets.

\section{Basin of Attraction in Dissipative Flows} 
\label{eq:basingen}

In the last of the previous section, we indicate that the time-average of an observable enables the visualization of low-dimensional slices through high-dimensional invariant sets of the quasiperiodically forced system with a smooth invariant measure.    
Here, we consider a more general class of the quasiperiodically forced systems \emph{with dissipation} and will point out that the use of time-average works for characterizing an important invariant set---\emph{basin of attraction}.

In Section~\ref{sec:ergodic} we considered the measure-preserving flow on a compact metric space.   
\emph{Every continuous flow} on a compact metric space preserves a measure.  
This is the content of the so-called Krylov-Bogolyubov theorem \cite{Sinai:1989}.  
However, this might be a useless statement if our intent is to study the behavior of trajectories from their time-averages, by the method used in the measure-preserving cases in Section~\ref{sec:ergodic}.   
For example, let us take a simple dissipative flow,
\[
\frac{{d} x}{{d} t}=-\lambda x \qquad
x\in\bbR, \qquad
\lambda>0.
\]  
All of the initial conditions converge to $x=0$ as $t\to\infty$ along the flow $S^t(x)=\exp(-\lambda t)x$.  
The invariant measure clearly exists and is the Dirac measure supported at $x=0$ (to which any ``initial" measure defined on $\mathbb{R}$ evolves).  
Thus, the following statement holds: for a function $f:\bbR\to\mathbb{C}$, 
\begin{equation}
f^\ast(x)=\int_{\mathbb{R}}f{d}\mu,
\label{eq:spacetime}
\end{equation}
for a.e. point $x\in\mathbb{R}$ with respect to the Dirac measure.  
However, that excludes the whole real line except for the origin itself.  
The situation feels better though.  
For example, it is easy to see that the time-average of any continuous function on $\bbR$ is just its value at $0$: $f^*(x)=f(0)$.  
Note that for the measure-preserving case in Section~\ref{sec:ergodic}, we basically considered integrable functions for which time-averages exist.  
We can not do that in the dissipative case.  
A function that is $1$ everywhere on a finite interval $(\underline{a},\overline{a})$ containing $0$, except at $0$ where its value is $0$, is clearly integrable (being a union of two simple functions). 
Then, its time-average under the dissipative flow is identically $1$, and its integral with respect to the invariant Dirac measure is $0$.  
It turns out that in dissipative systems, it is best for the time-average to work with continuous functions.

Adopting that idea, the whole method of ergodic partitioning can be extended to capture basins of attraction for continuous flows that are not necessarily measure-preserving. 
Consider a continuous flow $\vct{S}^t: X\to X$, where $X$ is not necessary compact, and label by $F$ the set on which time averages of continuous functions do not exist. 
On the set $X\setminus F$, consider a set $C$ on which time-averages of continuous functions (or, better, a countable, separating set of continuous functions) are constant.\footnote{There can be uncountably many such sets inside $X\setminus F$: for example, let $\vct{S}^t$ be the identity map, mapping every point into itself.} 
Then, by the same construction described in Section~\ref{sec:ergodic} and Appendix~\ref{appsec:2}, there is an invariant measure $\mu_C$ such that
\begin{equation}
f^*(\vct{x})=\int_C f d\mu_C
\label{eq:spacetime1}
\end{equation}
for any continuous function $f$, $\vct{x}\in C$. Such a measure is called a physical measure \cite{Young_JSP108}, provided $M$ is equipped with a measure $\mu$ that $\vct{S}^t$ does not preserve (but we are interested in)---say Lebesgue measure---and $C$ contains an open (and thus positive measure) set in $M$.\footnote{In light of this discussion, one might say that the notion of ``ergodic measure" should be preserved for such constructs even in dissipative systems. 
Namely, many ergodic measures we obtain in measure-preserving system are certainly physically important, although their support does not contain an open set. 
The tricky part is that due to Birkhoff's Ergodic Theorem, ergodic measures are associated with time averages of  integrable functions. In dissipative systems, this does not work (see the above example). 
We could keep the requirement of equality of space and time averages, like in \eqref{eq:spacetime}, but for continuous functions only. 
Physical measures would then be ergodic measures with support that contains an open set.} 
Suppose that the state space of a continuous flow admits a finite number of attractors, and that the union of their basins of attraction is of full measure. 
We state the result for flows (continuous-time systems) to emphasize that considerations here work for both discrete and continuous time.
\begin{theorem}
\label{thm:basin1}
Let the system
\[
\left.\frac{d\vct{S}^t(\vct{x})}{dt}\right|_{t=0}=\vct{F}(\vct{x}), \quad \vct{x}\in X\subset\bbR^n
\]
with a flow $\vct{S}^t$ have a finite number $N$ of attractors with basins of attraction $A_j, j=1,...,N$, such that $\mu(\cup_j A_j)=\mu(X)$, where $\mu$ is the Lebesgue measure. 
Also, let the time averages of continuous functions exist everywhere on a set $X\setminus F$, where $\mu(F)=0$. 
Then, the time-average $h^*(\vct{x})$ of a continuous function $h\in\mathcal{C}(X)$ is a piecewise constant eigenfunction of the Koopman operators $U^t$ at eigenvalue $0$ defined a.e. point with respect to $\mu$.
\end{theorem}
\begin{proof}
For any point $\vct{x}\in A_j,$ there exists a set $C$ such that
\begin{align}
(U^t h^*)(\vct{x})&=U^t \lim_{T\to \infty}\frac{1}{T}\int_0^{T} h(\vct{S}^\tau(\vct{x}))d\tau\nonumber \\
&= \int_C (U^t h) d\mu_C\nonumber \\
&=\int_C (h\circ\vct{S}^t) d\mu_C\nonumber \\
&=\int_C h d\mu_C \nonumber \\
&=h^*(\vct{x}), \nonumber
\end{align}
where in the second line we utilized \eqref{eq:spacetime1}, and transitioning from line 3 to 4 we used the fact that $\mu_C$ is invariant under $\vct{S}^t$. 
Since
\[
(U^t h^*)(\vct{x})=h^*(\vct{x})
\]
is exactly the equation which $h^*(\vct{x})$ has to satisfy in order to be an eigenfunction of $U^t$, and the basin of attraction $A_j$ is arbitrary, by taking into account that $\mu(\cup_j A_j)=\mu(X)$  the proposition is proven.
\end{proof}

The concept of the ergodic partition in dissipative systems is applied to a more general class of systems than those treated in Theorem~\ref{thm:basin1}. 
It provides ergodic measures on invariant sets that are not necessarily attractors---because they do not have an open set of initial conditions converging to them. 
Yet, these sets are dynamically important and can not be refined without losing the equivalence of space and time averages.

More precisely, the ergodic partition $\zeta\sub{e}$ of $M$ under $\vct{S}^t$ (not necessarily measure-$\mu$ preserving) is a partition into sets $C_{\alpha}$, where $\alpha$ is a member of an indexing set, such that on each set $C_{\alpha}$ (ergodic set) there exists an ergodic measure $\mu_{C_\alpha}$ such that
\begin{enumerate}
\item $\mu_{C_{\alpha}}(C_{\alpha})=1,$
\item  For every $f\in\mathcal{C}(M),f^{\ast}(\vct{x}\in C_{\alpha})=\int_{C_{\alpha}%
}fd\mu_{C_{\alpha}},$  for a.e $\vct{x}$ point with respect to $\mu$.
\end{enumerate}
The last condition emphasizes the role of the measure $\mu$---the time averages are equal to space averages with respect to $\mu_{C_\alpha}$, but their equality is almost everywhere with respect to measure $\mu$ that might be of our interest. 
In this way, we have escaped the realm of the Krylov-Bogolyubov Theorem, which, in the case when dynamics are not measure-preserving, neglects dynamics on possibly large swaths of the state space.

\section{Sample-Based Characterization of Invariant Sets}
\label{sec:chara}

In this section, we provide a characterization of invariant sets in the quasiperiodically forced system \eqref{eqn:org_syst}.  
Generally speaking, understanding invariant structures of \eqref{eqn:org_syst} is not easy because the associated 
portrait of $M$ changes with time in an aperiodic manner.  
As shown in Section~\ref{sec:ergodic}, the ergodic partition theory enables one to visualize an invariant set as its low-dimensional slice in $M$ 
by the time-averaging technique.  
Also, in \eqref{eq:basingen} it was shown that the time-averaging technique works for dissipative case to visualize the basin of attraction.  
The obtained slice here is just a sample of the invariant set at a \emph{particular} initial time (in other words, an initial condition in $\bbT^N$).  
Because of the aperiodic nature, such a sample does not seem to provide complete information on the entire structure of invariant set in the augmented state space $M\times\bbT^N$.  
However, we will prove that a boundedness property of the invariant set in $M\times\bbT^N$ is captured by means of one sample of it in $M$ (see Theorem~\ref{thm:2-IM}).  
In the following, in order to encompass both the cases in Section~\ref{sec:ergodic} and \eqref{eq:basingen}, we consider the original state space $M$ that is metric and not necessarily compact.  

A  continuous, linear skew-product flow $\vct{S}^t$ (diffeomorphism) on the augmented state space 
$X:=M\times\bbT^N$ derived from the quasiperiodically forced system \eqref{eqn:org_syst} is described below:  for $\vct{x}=(\vct{m},\vct{\theta})\in M\times\bbT^N$,
\begin{equation}
\vct{S}^t(\vct{x}):=
\left(\vct{S}^t_{M}(\vct{m},\vct{\theta}),\vct{S}^t_\mathit{\Omega}(\vct{\theta})\right), 
\end{equation}
where $\vct{S}^t_{M}: M\times\bbT^N\to M$ is the continuous map defined by trajectories of the original system \eqref{eqn:org_syst}, and $\vct{S}^t_\mathit{\Omega}: \bbT^N\to\bbT^N$ the linear (continuous) flow on the torus $\bbT^N$. 
For a fixed $\vct{\theta}_0\in\bbT^N$, 
we will write the orbit on $\bbT^N$ through $\vct{\theta}_0$ as $\mO_{\bbT^N}(\vct{\theta}_0):=\{\vct{\theta} : \vct{\theta}=\vct{S}^t_\mathit{\Omega}(\vct{\theta}_0)\in\bbT^N, t\in\bbR\}$.  
Since the $N$ basic frequencies in the original system \eqref{eqn:org_syst} are rationally independent, the following lemma is obvious.  
\begin{lemma}
\label{lemma:dense}
For any $\vct{\theta}_0\in\mathbb{T}^N$, $\mO_{\bbT^N}(\vct{\theta}_0)$ is dense in $\bbT^N$.  
\end{lemma}

Now, we investigate an invariant set in the augmented system \eqref{eqn:aug_syst}.  
Let $I\subseteq X$ be a positively invariant set of \eqref{eqn:aug_syst} that is supposed to be closed\footnote{If $I$ is invariant and open, then its closure $\cl(I)$ is still invariant.}.  
It follows from Lemma~\ref{lemma:dense} and by closure that the following decomposition of $I$ holds:
\begin{equation}
I=\bigcup_{\theta\in \bbT^N} A_\theta\times\{\vct{\theta}\},
\label{eqn:dI}
\end{equation}
where $A_{\theta}$ is a closed subset of $M$.  
Let us denote by $A_\thz$ an intersection or sample of the invariant set $I$ at $\vct{\theta}=\vct{\theta}_0$. 
The next lemma provides a topological property of the intersections $A_\theta$. 
\begin{lemma}
\label{lemma:topo}
For each $\epsilon>0$ there exists a positive constant $\delta$ so that $|\vct{\theta}-\vct{\theta}_0|_{\mathbb{T}^N}<\delta$ implies $d(A_\theta,A_\thz)<\epsilon$, where $d(A_{\theta},A_\thz)$ is the Hausdorff distance that induces a topology on the family of all closed subsets of $M$. 
\end{lemma}
\begin{proof}
Assume not.  
Then, for each $\delta>0$ there exist a sequence of times $\{t_j\}$, where $t_j\to\infty $ as $j\to\infty$, and a positive constant $\epsilon_1$ such that $\vct{\theta}_j=\vct{S}^{t_j}_\mathit{\Omega}(\vct{\theta}_0)$ satisfies $|\vct{\theta}_j-\vct{\theta}_0|_{\mathbb{T}^N}<\delta$ for every $j>J$ ($J$ is an integer), while $d(A_{\theta_k},A_\thz)>\epsilon_1$ holds for some $k>J$. 
Here, from the continuity of $\vct{S}^t_M$ in $t$, for each $\vct{m}\in A_\thz$ and $\epsilon>0$, there exists a positive constant $\delta$ such that $|\vct{\theta}_k-\vct{\theta}_0|_{\mathbb{T}^N}<\delta$ implies $|\vct{m}_k-\vct{m}|_{M}<\epsilon$, where $\vct{m}_k:=\vct{S}^{t_k}_{M}(\vct{m},\vct{\theta}_0)\in A_{\theta_k}$. 
This gives us a contradiction of $d(A_{\theta_k},A_\thz)>\epsilon_1$ by taking $\epsilon=\epsilon_1$ and from the definition of Hausdorff distance.
\end{proof}

This lemma suggests that the invariant set $I$ is topologically connected with respect to $\vct{\theta}$. 
One trivial example is provided by taking as $I$ the closure of a single trajectory starting at a point $(\vct{m},\vct{\theta}_0)$, for which $A_\thz$ consists of a single point.   
In this case, $I$ is topologically regarded as a product set of a single point in $M$ and the torus $\mathbb{T}^N$. 
This is rigorously stated in the next theorem. 
\begin{theorem}
\label{ther:qpsets-IM}
Suppose that $A_\thz$ is bounded in $M$. 
Then, $A_\thz\times \mathbb{T}^N$ is homeomorphic to $I$.
\end{theorem}
\begin{proof}
Now, let us construct the mapping $h_\thz: A_\thz\times \mathcal{O}_{\mathbb{T}^N}(\vct{\theta}_0)\to I$ in the following manner. 
For each $(\vct{m},\vct{\theta})\in A_\thz\times\mathcal{O}_{\mathbb{T}^N}(\vct{\theta}_0)$, by choosing time $\tau_\theta$ that satisfies $\vct{\theta}=\vct{S}^{\tau_\theta}_{\it\Omega}(\vct{\theta}_0)$, we define $h_\thz(\vct{m},\vct{\theta})$ as $(\vct{S}^{\tau_\theta}_M(\vct{m},\vct{\theta}_0),\vct{\theta})$. 

We here prove the continuity of $h_\thz$. 
For each $\epsilon>0$ and $\vct{\theta}\in\mathcal{O}_{\mathbb{T}^N}(\vct{\theta}_0)$, there exists a sequence of $\{t_j\}$ (where $t_j\to\infty$ as $j\to\infty$) such that $\vct{\theta}_j=\vct{S}^{t_j}_\mathit{\Omega}(\vct{\theta}_0)$ satisfies $\left|\vct{\theta}_j-\vct{\theta}\right|_{\mathbb{T}^N}<\epsilon/2$ for every $j>J$ ($J$ is an integer).  
Furthermore, because of the continuity of $\vct{S}^t_{M}$ in $t$, for each $\vct{m}_j\in A_\thz$ there exists a positive constant $\delta_j$ such that $\left|\vct{m}-\vct{m}_j\right|_M<\delta_j$ implies $\left|\vct{S}_M^{\tau_\theta}(\vct{m},\vct{\theta}_0)-\vct{S}_M^{\tau_{\theta_j}}(\vct{m}_j,\vct{\theta}_0)\right|_M<\epsilon/2$. 
Here, because for every $j>J$
\[
\left|(\vct{m},\vct{\theta})-(\vct{m}_j,\vct{\theta}_j)\right|_{M\times\mathbb{T}^N}
\leq \left|\vct{m}-\vct{m}_j\right|_{M}+\left|\vct{\theta}-\vct{\theta}_j\right|_{\mathbb{T}^N}
< \delta_j+\frac{\epsilon}{2},
\]
we set $\delta:=\delta_j+\epsilon/2$. 
Thus, if $\left|(\vct{m},\vct{\theta})-(\vct{m}_j,\vct{\theta}_j)\right|_{M\times\mathbb{T}^N}<\delta$, then
\begin{align}
\left|h_\thz(\vct{m},\vct{\theta})-h_\thz(\vct{m}_j,\vct{\theta}_j)\right|_{M\times\mathbb{T}^N}
&\leq \left|\vct{S}_M^{\tau_\theta}(\vct{m},\vct{\theta}_0)-\vct{S}_M^{\tau_{\theta_j}}(\vct{m}_j,\vct{\theta}_0)\right|_M
+\left|\vct{\theta}-\vct{\theta}_j\right|_{\mathbb{T}^N}
\nonumber\\
& <\frac{\epsilon}{2}+\frac{\epsilon}{2}
=\epsilon.
\nonumber
\end{align}
This shows that $h_\thz$ is continuous.   

We next prove that $h_\thz$ is uniformly continuous. 
Assume not. 
Then, for each $\delta>0$ there exists a positive constant $\epsilon_1$ such that $\left|(\vct{m}'_j,\vct{\theta}'_j)-(\vct{m}_j,\vct{\theta}_j)\right|_{M\times\mathbb{T}^N}<\delta$ implies\\
$\left|h_\thz(\vct{m}'_j,\vct{\theta}'_j)-h_\thz(\vct{m}_j,\vct{\theta}_j)\right|_{M\times\mathbb{T}^N}\geq \epsilon_1$ for some $(\vct{m}_j,\vct{\theta}_j) $ and $(\vct{m}'_j,\vct{\theta}'_j)\in A_\thz\times\mathcal{O}_{\mathbb{T}^N}(\vct{\theta}_0)$.  
We here set $\delta=1/j$. 
Since it is supposed that $A_\thz$ is bounded and closed (and $\mathcal{O}_{\mathbb{T}^N}(\vct{\theta}_0)\subset \mathbb{T}^N$), the two sequences of points $\{(\vct{m}_j,\vct{\theta}_j)\}$ and $\{(\vct{m}'_j,\vct{\theta}'_j)\}$ have convergent subsequences, denoted as $\{(\vct{m}_{j_k},\vct{\theta}_{j_k})\}$ and $\{(\vct{m}'_{j_k},\vct{\theta}'_{j_k})\}$.  
Because of $\left|(\vct{m}'_j,\vct{\theta}'_j)-(\vct{m}_j,\vct{\theta}_j)\right|_{M\times\mathbb{T}^N}<1/j$, both the subsequences have the same convergent point, denoted as $(\vct{m}^\ast,\vct{\theta}^\ast)$.  
That is, for each $\delta_1>0$, there exists a positive integer $J_k$ such that both the inequalities $\left|(\vct{m}'_{j_k},\vct{\theta}'_{j_k})-(\vct{m}^\ast,\vct{\theta}^\ast)\right|_{M\times\mathbb{T}^N}<\delta_1$ and $\left|(\vct{m}_{j_k},\vct{\theta}_{j_k})-(\vct{m}^\ast,\vct{\theta}^\ast)\right|_{M\times\mathbb{T}^N}<\delta_1$ hold for every $j_k\geq J_k$. 
Here, since $h_\thz$ is proven to be continuous at $(\vct{m}^\ast,\vct{\theta}^\ast)$, for each $\epsilon>0$, there exists a positive constant $\delta_2$ such that $\left|(\vct{m},\vct{\theta})-(\vct{m}^\ast,\vct{\theta}^\ast)\right|_{M\times\mathbb{T}^N}<\delta_2$ implies $\left|h_\thz(\vct{m},\vct{\theta})-h_\thz(\vct{m}^\ast,\vct{\theta}^\ast)\right|_{M\times\mathbb{T}^N}<\epsilon/2$. 
By taking $\epsilon=\epsilon_1$ and $\delta_1=\delta_2$ thus choosing $J_k$ appropriately, we see that both the inequalities\\
$\left|h_\thz(\vct{m}'_{j_k},\vct{\theta}'_{j_k})-h_\thz(\vct{m}^\ast,\vct{\theta}^\ast)\right|_{M\times\mathbb{T}^M}<\epsilon_1/2$ and $\left|h_\thz(\vct{m}_{j_k},\vct{\theta}_{j_k})-h_\thz(\vct{m}^\ast,\vct{\theta}^\ast)\right|_{M\times\mathbb{T}^N}<\epsilon_1/2$ hold for every $j_k\geq J_k$, and
\[
\begin{aligned}
\left|h_\thz(\vct{m}'_{j_k},\vct{\theta}'_{j_k})-h_\thz(\vct{m}_{j_k},\vct{\theta}_{j_k})\right|_{M\times\mathbb{T}^M}
\leq &
\left|h_\thz(\vct{m}'_{j_k},\vct{\theta}'_{j_k})-h_\thz(\vct{m}^\ast,\vct{\theta}^\ast)\right|_{M\times\mathbb{T}^N} \\
& + \left|h_\thz(\vct{m}^\ast,\vct{\theta}^\ast)-h_\thz(\vct{m}_{j_k},\vct{\theta}_{j_k})\right|_{M\times\mathbb{T}^N} \\
< & \frac{\epsilon_1}{2}+\frac{\epsilon_1}{2} = \epsilon_1,
\end{aligned}
\]
implying the contradiction for $\left|h_\thz(\vct{m}'_{j_k},\vct{\theta}'_{j_k})-h_\thz(\vct{m}_{j_k},\vct{\theta}_{j_k})\right|_{M\times\mathbb{T}^N}\geq \epsilon_1$. 
Thus, $h_\thz$ is uniformly continuous. 

By virtue of Theorem~3.45 in page~136 of \cite{Kubrusly:2001}, the uniformly-continuous $h_\thz: A_\thz\times\mathcal{O}_{\mathbb{T}^N}(\vct{\theta}_0)\to I$ is extended to the mapping $\hat{h}_\thz: A_\thz\times\mathbb{T}^N\to I$ that is also (uniformly) continuous. 

We now prove that $\hat{h}_\thz$ is a bijection.  
For each $\vct{\theta}\in\mathcal{O}_{\mathbb{T}^N}(\vct{\theta}_0)$, the statement is obvious from the construction of $h_\thz$ (see its dependence on $\vct{\theta}$).  
It is thus enough to check the case $\vct{\theta}\in\mathbb{T}^N\setminus \mathcal{O}_{\mathbb{T}^N}(\vct{\theta}_0)$. 
For each $\vct{\theta}\in\mathbb{T}^N\setminus \mathcal{O}_{\mathbb{T}^N}(\vct{\theta}_0)$, there exists a sequence of times $\{t_j\}$ (where $t_j\to\infty $ as $j\to\infty$) such that $\vct{\theta}_j=\vct{S}^{t_j}_\mathit{\Omega}(\vct{\theta}_0)$ converges to $\vct{\theta}$.  
Because $\hat{h}_\thz$ is continuous, the sequence $\{\hat{h}_\thz(\vct{m},\vct{\theta}_j)=h_\thz(\vct{m},\vct{\theta}_j)\}$ converges to $\hat{h}_\thz(\vct{m},\vct{\theta})$, which is represented as $\displaystyle \left(\lim_{t_j\to\infty}\vct{S}_M^{t_j}(\vct{m},\vct{\theta}_0),\vct{\theta}\right)$. 
Namely, the limit exists for every $\vct{m}\in A_\thz$. 
Assume that $\hat{h}_\thz$ is not an injection. 
Then, there exist $\vct{m},\vct{m}'\in A_\thz$ satisfying $\vct{m}\neq\vct{m}'$ such that for each $\epsilon>0$, $\left|\vct{S}_M^{t_j}(\vct{m},\vct{\theta}_0)-\vct{S}_M^{t_j}(\vct{m}',\vct{\theta}_0)\right|_M<\epsilon$ holds for every $j>J$ ($J$ is an integer and depends on $\epsilon$). 
Here, the fundamental property of uniqueness of trajectories in \eqref{eqn:org_syst} (or $\vct{S}^t$ diffeomorphsim) implies that for each $t_j$, $\left|\vct{S}_M^{t_j}(\vct{m},\vct{\theta}_0)-\vct{S}_M^{t_j}(\vct{m}',\vct{\theta}_0)\right|_M\geq\epsilon_j$ holds ($\epsilon_j$ is a positive constant and depends on $t_j$), showing the contradiction by taking $\epsilon=\epsilon_j$. 
Hence, $\hat{h}_\thz$ is an injection.  
Regarding $\hat{h}_\thz$ surjective, it is obvious that for each $\vct{\theta}\in\mathcal{O}_{\mathbb{T}^N}(\vct{\theta}_0)$, $\hat{h}_\thz(A_\thz,\vct{\theta})=h_\thz(A_\thz,\vct{\theta})=\left(\vct{S}_M^{\tau_\theta}(A_\thz,\vct{\theta}_0),\vct{\theta}\right)=(A_\theta,\vct{\theta})$. 
For each $\vct{\theta}\in\mathbb{T}^N\setminus \mathcal{O}_{\mathbb{T}^N}(\vct{\theta}_0)$ we see $\displaystyle \hat{h}_\thz(A_\thz,\vct{\theta})=\left(\lim_{t_j\to\infty}\vct{S}_M^{t_j}(A_\thz,\vct{\theta}_0),\vct{\theta}\right)$, and the limit converges to $A_\theta$ from Lemma~\ref{lemma:topo}. 
Hence, from  \eqref{eqn:dI}, $\hat{h}_\thz$ is a surjection. 

Finally, the inverse of $\hat{h}_\thz$ exists according to the above construction and is continuous. 
Therefore, $\hat{h}_\thz: A_{\theta_0}\times\mathbb{T}^N\to I$ is a homeomorphism, namely, $A_\thz\times\mathbb{T}^N$ is homeomorphic to $I$.
\end{proof}

The following two corollaries provide a way of characterization of boundedness of $I$ by means of one sample of it.  
\begin{corollary}
\label{thm:2-IM}
Suppose that $A_\thz$ is bounded in $M$.  
Then, $I$ is bounded in $X$.
\end{corollary}
\begin{corollary}
\label{thm:NEW}
Suppose that a closed subset $B_\thz$ of $A_\thz$ is bounded in $M$. 
Then, $\hat{h}_\thz(B_\thz\times\bbT^N)$ is a subset of $I$ and bounded in $X$.
\end{corollary}

Corollaries~\ref{thm:2-IM} and \ref{thm:NEW} imply that by computing a slice of the invariant set in $M$ at one sample onset, 
it is possible to determine whether the invariant set or its subset is bounded in the augmented state space $X$.  
For a bounded invariant set $I$ in $X$, any cross-section $I_\thz=A_\thz\times \{\vct{\theta}_0\}$ 
($\vct{\theta}_0\in\bbT^N$) is bounded in terms of $M$, 
in other words, the boundedness property of $A_\thz$ does not depend on the choice of $\vct{\theta}_0$.     
In this way, we call the bounded invariant set $I$ in $X$ \emph{uniformly} bounded invariant set.

\section{Example Studies}
\label{sec:example}

In this section, we provide three examples of the application of the above theoretical results.  
The last example is related to practical problems on stability of power grids.  

\subsection{Forced Linear Harmonic Oscillator}
\label{subsec:ex1}

First, we will consider the following one-degree-of-freedom linear harmonic oscillator with quasi-periodic forcing of (multiple) $N$ frequencies:
\[
\frac{dm_1}{dt}=m_2, \qquad
\frac{dm_2}{dt}=-m_1
+\sum^{N}_{i=1}F_i\sin(\mathit{\Omega}_it+\theta_{i0}),
\]
or
\begin{equation}
\frac{dm_1}{dt}=m_2, \qquad
\frac{dm_2}{dt}=-m_1+\sum^{N}_{i=1}F_i\sin\theta_i, \qquad
\frac{d\theta_i}{dt}=\mathit{\Omega}_i \qquad i=1,2,\ldots,N
\label{eqn:ex1}
\end{equation}
where $F_i$ ($>0$) ($i=1,2,\ldots,N$) are the amplitudes of periodic forces, $\mathit{\Omega}_i$ ($>0$) the angular frequencies of the forces, and $\theta_{i0}\in\mathbb{T}$ the initial phases.  
We assume there is no resonance: the $N$ angular frequencies $\mathit{\Omega}_i$ are rationally independent; and $\mathit{\Omega}_i\neq 1$ ($i=1,2,\ldots,N$).  
The solution of \eqref{eqn:ex1} with initial condition $(m_{10},m_{20},\theta_{10},\theta_{20},\ldots,\theta_{N0})$ is analytically represented as follows:
\[
\left.
\begin{aligned}
m_1(t) &= 
C_1\cos t+C_2\sin t
+\sum^{N}_{i=1}\frac{F_i}{1-\mathit{\Omega}^2_i}\cos(\mathit{\Omega}_it+\theta_{i0}), \\
m_2(t) &=
\displaystyle \frac{{d}}{{d} t}m_1(t), \\\noalign{\vskip 5mm}
\theta_i(t) &= \mathit{\Omega}_it+\theta_{i0} \qquad i=1,2,\ldots,N,
\end{aligned}
\right\}
\]
with
\begin{equation}
C_1:=
m_{10}-\sum^{N}_{i=1}\frac{F_i}{1-\mathit{\Omega}^2_i}\cos\theta_{i0}, \qquad
C_2:= 
m_{20}+\sum^{N}_{i=1}\frac{F_i\mathit{\Omega}_i}{1-\mathit{\Omega}^2_i}\sin\theta_{i0}.
\label{eqn:ex1_init}
\end{equation}

Now, we estimate the time-averages of an observable under the dynamics described by the linear system \eqref{eqn:ex1}.  
We define the following quadratic observable $f: \bbR^2\times \bbT^N\to\bbR$, related to the potential energy $m^2_1/2$:
\[
f(m_1,m_2,\theta_1,\theta_2,\ldots,\theta_N)=m^2_1.
\]
Then, we calculate its time-average under the solution as
\begin{equation}
\lim_{T\to\infty}\frac{1}{T}\int^T_0(U^tf)(m_{10},m_{20},\theta_{10},\theta_{20},\ldots,\theta_{N0}){d} t
=\frac{1}{2}
\left\{C^2_1+C^2_2+\sum^{N}_{i=1}\frac{F^2_i}{(1-\mathit{\Omega}^2_i)^2}\right\}.
\end{equation}
Here, as shown in \eqref{eqn:ex1_init}, the constants $C_1$ and $C_2$ are determined by the initial condition $(m_{10},m_{20},\theta_{10},\theta_{20},\ldots,\theta_{N0})$.  
The level set, parameterized by a real-valued constant $c\in\mathbb{R}$, of the time-average in the two-dimensional initial plane $(m_{10},m_{20})$ at fixed $(\theta_{10},\theta_{20},\ldots,\theta_{N0})$ becomes a circle with center of
\[
\left(\sum^{N}_{i=1}\frac{F_i}{1-\mathit{\Omega}^2_i}\cos\theta_{i0},-\sum^{N}_{i=1}\frac{F_i\mathit{\Omega}_i}{1-\mathit{\Omega}^2_i}\sin\theta_{i0}\right),
\]
if the following inequality holds:
\[
r^2:=2c-\sum^{N}_{i=1}\frac{F^2_i}{(1-\mathit{\Omega}^2_i)^2}>0
\]
Thus, the constant $r$ corresponds to the radius of the circle of the level set for $c$.  
The center of the level set seems to stay close to the origin and is shifted with the choice of initial phases $(\theta_{10},\theta_{20},\ldots,\theta_{N0})$.
Numerical results of the level sets for $N=2$, $\mathit{\Omega}_1=\pi/3$, $\mathit{\Omega}_2=11/10$, and $F_1=F_2=0.2$ are presented in Figure~\ref{fig:0}. 

\begin{figure}[t]
\centering
\includegraphics[width=0.6\textwidth]{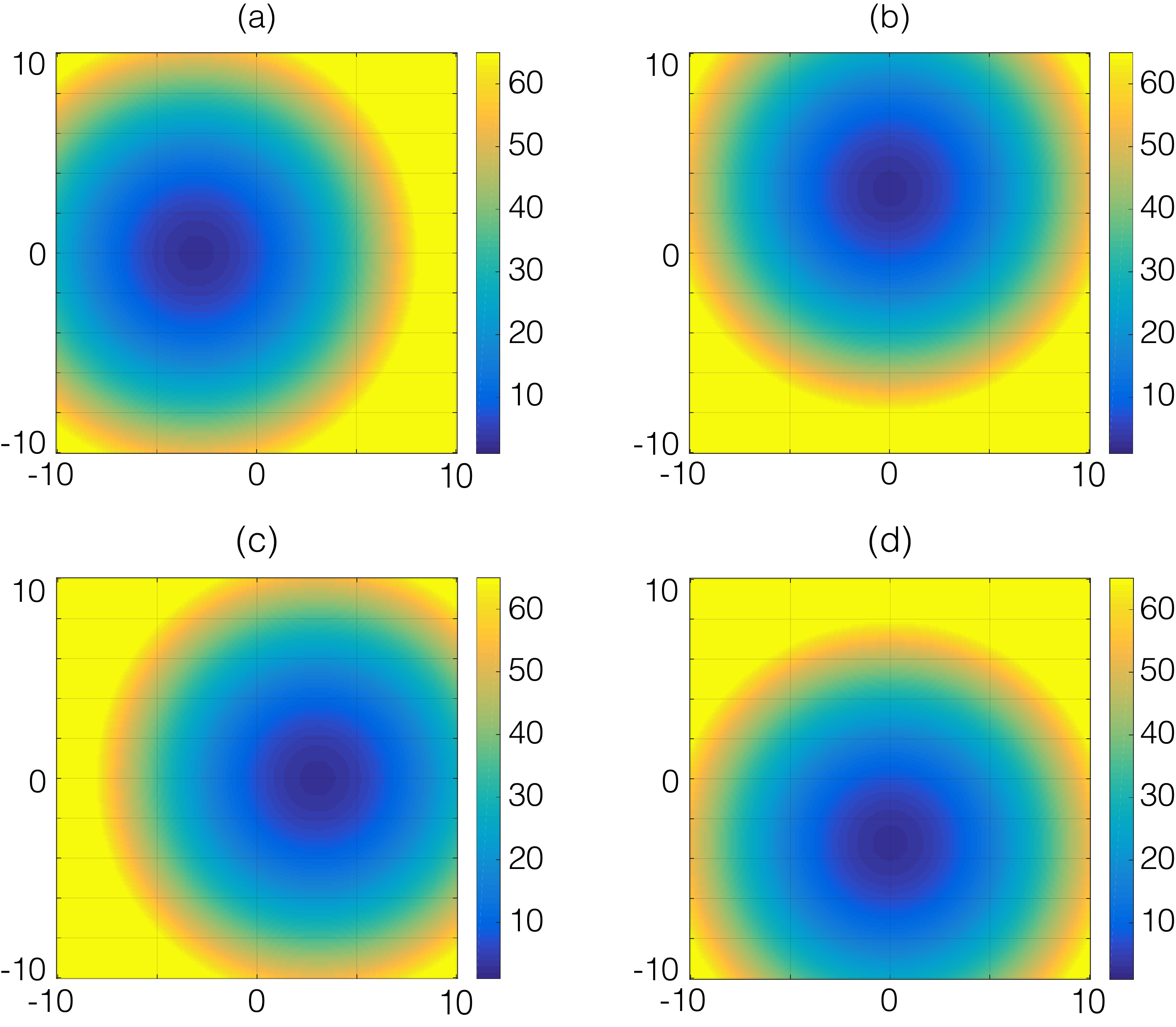}
\caption{%
Level sets of the quasiperiodically-forced one-degree-of-freedom harmonic oscillator \eqref{eqn:ex1} for different $(\theta_{10},\theta_{20})$: 
(a) $(0,0)$, (b) $(\pi/2,\pi/2)$, (c) $(\pi,\pi)$, and (d) $(3\pi/2,3\pi/2)$.  
The horizontal (or vertical) axis for each figure is $x_1\in[-10,10]$ (or $x_2\in[-10,10]$).
}%
\label{fig:0}
\end{figure}

\subsection{One-Dimensional Dissipative System}
\label{subsec:ex2}

Next, we consider the following linear dissipative system with quasi-periodic forcing:
\[
\frac{{d} m}{{d} t}=-\lambda m
+\sum^{N}_{i=1}F_i\sin(\mathit{\Omega}_it+\theta_{i0}).
\]
or
\begin{equation}
\frac{dm}{dt}=-\lambda m+\sum^{N}_{i=1}F_i\sin\theta_i, \qquad 
\frac{d\theta_i}{dt}=\mathit{\Omega}_i \qquad i=1,2,\ldots,N
\label{eqn:ex2}
\end{equation}
where $\lambda>0$.  
In the same manner as above, we suppose the $N$ angular frequencies $\mathit{\Omega}_i$ are rationally independent.  
The solution of \eqref{eqn:ex2} with initial condition $(m_{0},\theta_{10},\theta_{20},\ldots,\theta_{N0})$ is analytically represented as follows:
\[
\left.
\begin{array}{ccl}
m(t) &=& \displaystyle C\ee^{-\lambda t}
+\sum^{N}_{i=1}\frac{F_i}{\lambda^2+\mathit{\Omega}^2_i}
\{\lambda\sin(\mathit{\Omega}_it+\theta_{i0})
-\mathit{\Omega}_i\cos(\mathit{\Omega}_it+\theta_{i0})\},
\\\noalign{\vskip 4mm}
\theta_i(t) &=& \mathit{\Omega}_it+\theta_{i0}~~~(i=1,2,\ldots,N),
\end{array}
\right\}
\]
with
\[
C:=m_0
-\sum^{N}_{i=1}\frac{F_i}{\lambda^2+\mathit{\Omega}^2_i}
\{\lambda\sin\theta_{i0}-\mathit{\Omega}_i\cos\theta_{i0}\}.
\]

Here, we estimate the time-averages of an observable under the dynamics described by the linear system \eqref{eqn:ex2}.  
We consider the quadratic observable as
\[
f(m,\theta_1,\theta_2,\ldots,\theta_N)=m^2.
\]
Then, its time-average under the solution is directly calculated as follows:
\[
\lim_{T\to\infty}\frac{1}{T}\int^T_0(U^tf)(m_0,\theta_{10},\theta_{20},\ldots,\theta_{N0}){d} t
=\frac{1}{2}
\sum^{N}_{i=1}\frac{F^2_i}{\lambda^2+\mathit{\Omega}^2_i}.
\]
Although this is trivial, since the transient term in the solution is filtered out, the value of the time-average is constant in the augmented state space $\mathbb{R}\times\mathbb{T}^N$, in other words, does not depend on the initial condition $(m_0,\theta_{10},\theta_{20},\ldots,\theta_{N0})$.  
This implies that for any choice of sample at $(\theta_{10},\theta_{20},\ldots,\theta_{N0})^\top\in\mathbb{T}^N$, the partition of $M=\mathbb{R}$ based on the level set of the time-averages is just one and corresponds to $\mathbb{R}$ itself.  

Also, the above observation on the partition holds when we pick up a polynomial observable like
\[
f(m,\theta_1,\theta_2,\ldots,\theta_N)=\sum^{n}_{j=0}a_jm^j,
\]
with coefficients $a_j\in\mathbb{R}$ and finite integer $n$.  
For every continuous function defined on a closed interval in $\mathbb{R}$, from the Weierstrass Approximation Theorem, it can be uniformly approximated on that interval by polynomials to any degree of accuracy.  
This implies that the above observation of  the partition holds for every continuous function, which is an example of Theorem~\ref{thm:basin1}.  
Thus, the augmented state space $\mathbb{R}\times\mathbb{T}^N$ is the basin of attraction for the system \eqref{eqn:ex2}, that is to say, the torus $\bbT^N$ is the unique attractor for the system.

\subsection{Two-Dimensional Nonlinear Model of a Power Grid}
\label{sec:appl}

To show the effectiveness of the theory beyond the basic examples, we apply the developed theory for analyzing 
the CSI phenomenon in a loop power grid.  
CSI is a undesirable and emergent phenomenon of synchronous machines in a power grid, in which most of the machines coherently lose synchronism with the rest of the grid after being subjected to a finite disturbance \cite{
Susuki_JNLS09}. 
In the case of small dissipation, this phenomenon generally does not happen upon an infinitesimally small perturbation around a steady operating state (stable equilibrium).\footnote{If we have no dissipation, nonlinear stability properties are not understood in the high-dimensional cases.}
However, it encompasses the situation when the grid's operating state escapes a predefined, positive-measure set around the equilibrium.  
In this way, the notion of instability that we address here is non-local.  
In \cite{Susuki_JNLS09}, we derived a reduced-order dynamical system that described averaged dynamics of machines in a simple loop power grid and explained the non-local instability.  
The reduced-order system has an quasiperiodic forcing (so it is non-autonomous), and its solutions define a measure-preserving flow.  
The goal of this section is to characterize the non-local instability by analyzing invariant sets of the quasiperiodically forced system,  which is introduced in \eqref{eqn:swing_eq} below.    

\subsubsection{Mathematical Model}

\begin{figure}[t]
\centering
\includegraphics[width=.6\textwidth]{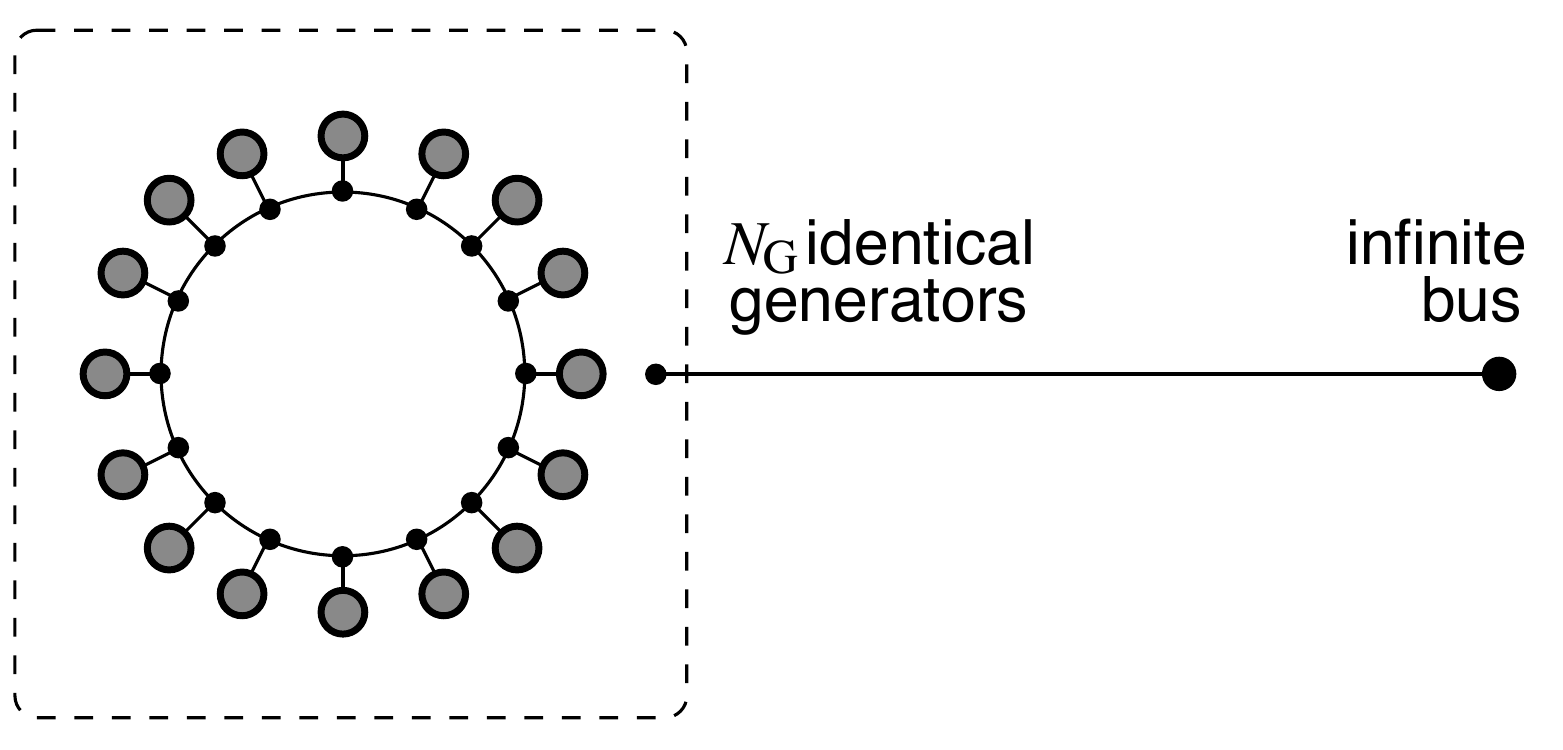}
\caption{%
Rudimentary power grid with the loop topology, named the loop power grid \cite{Susuki_JNLS09}
}%
\label{fig:loop}
\end{figure}

Consider short-term (zero to ten seconds) swing dynamics of a rudimentary power grid 
with the loop topology shown in Figure~\ref{fig:loop}, where the small \emph{gray} circles denote synchronous generators.  
The loop part of the grid consists of $N\sub{G}$ small, identical generators, encompassed by the \emph{dotted} box, which operate in the grid and are connected to the infinite bus\footnote{An ideal voltage source of constant voltage and constant frequency}.  
The loss-less transmission lines joining the infinite bus and a generator are much longer than those joining two generators in the loop part.  
Thus, the magnitude of interaction between the infinite bus and a generator is much smaller than that between two neighboring generators on the loop part.  
We call the model of Figure~\ref{fig:loop} the loop power grid in the following.  

Here, we assume that the lengths of transmission lines between two neighboring generators are identical.  
In \cite{Susuki_JNLS09}, we showed that the CSI phenomenon in the loop power grid was accurately captured by the following dynamical system defined on the cylindrical state space $\bbT^1\times\bbR$:
\begin{equation}
\frac{d\delta}{dt}=
\omega, \qquad
\frac{d\omega}{dt}=
p\sub{m}-\frac{b}{N\sub{G}}\sum_{i=1}^{N\sub{G}}
\sin\left(\sum_{j\in\mathcal{J}}e_{ij}c_j\cos\itO_jt+\delta\right)
\label{eqn:swing_eq}
\end{equation}
where
\[
e_{ij}=
\sqrt{\frac{2}{N\sub{G}}}
\cos\left(\frac{2\pi ij}{N\sub{G}}+\frac{\pi}{4}\right),
\qquad
\itO_j=
2\sqrt{|b\sub{int}|}\left|\sin\frac{\pi j}{N\sub{G}}\right|.
\]
The system \eqref{eqn:swing_eq} represents spatially-averaged dynamics of the $N\sub{G}$ generators in the loop power grid.  
The variable $\delta\in \bbT^1$ is the average of angular positions of rotors (with respect to the infinite bus) of the $N\sub{G}$ generators, and $\omega\in\bbR$ is the average of deviations of rotor speeds in the $N\sub{G}$ generators relative to the system angular frequency.  
The parameter $p\sub{m}$ stands for the mechanical input power to a generator, $b$ for the maximum transmission power between the infinite bus and a generator, and $b\sub{int}$ for the maximum transmission power between two neighboring generators in the loop power grid.  
The constants $e_{ij}$ are the eigenfunctions of linear modal oscillations between coupled generators in the loop part, $\itO_j$ their eigen-(angular) frequency, and $c_j$ the strengths of modal oscillations.   
The finite index set $\mathcal{J}$ determines which modes are excited in the loop part. 
The system \eqref{eqn:swing_eq} is derived under the observation that the linear modal oscillations in the loop part act as perturbations on the spatially-averaged dynamics of the $N\sub{G}$ generators: see \cite{Susuki_JNLS09} and references therein.    
Note that \eqref{eqn:swing_eq} is the Hamiltonian system:
\begin{equation}
\frac{d\delta}{dt}=\frac{\DD}{\DD\omega}{H}(\delta,\omega,t), \qquad
\frac{d\omega}{dt}=-\frac{\DD}{\DD\delta}{H}(\delta,\omega,t),
\label{eqn:time-H}
\end{equation}
with the time-dependent Hamiltonian function ${H}(\delta,\omega,t)$, given by
\[
{H}(\delta,\omega,t):=
\frac{1}{2}\omega^2-p\sub{m}\delta-\frac{b}{N\sub{G}}\sum^{N\sub{G}}_{i=1}\cos\left(\sum_{j\in\mathcal{J}}e_{ij}c_j\cos\itO_jt+\delta\right).
\]
Because the flow defined here is divergence-free, i.e. $(\DD{H}/\DD\delta)({d}\delta/{d} t)+(\DD{H}/\DD\omega)({d}\omega/{d} t)=0$, the system \eqref{eqn:swing_eq} preserves the Liouville measure ${d}\delta{d}\omega$.  
Note that it does not conserve the Hamiltonian function ${H}$, because of ${d}{H}/{d} t\neq 0$ if $c_j\neq 0$. 
In this way, the augmented system for \eqref{eqn:swing_eq}, namely
\begin{equation}
\frac{d\delta}{dt}=
\omega, \quad
\frac{d\omega}{dt}=
p\sub{m}-\frac{b}{N\sub{G}}\sum_{i=1}^{N\sub{G}}
\sin\left(\sum_{j\in\mathcal{J}}e_{ij}c_j\cos\theta_j+\delta\right), \quad
\frac{d\theta_j}{dt}=\itO_j \qquad j\in\mathcal{J},
\label{eqn:swing_eq-t}
\end{equation}
defines a measure-preserving flow on $M\times\bbT^{|{\cal J}|}$ with $M=\bbT^1\times\bbR$, where $|{\cal J}|$ stands for the cardinality of ${\cal J}$. 

\subsubsection{Results and Implications}
\label{subsec:vis}
 
It was shown in \cite{Susuki_JNLS09} that the unbounded motion in the quasiperiodically forced system \eqref{eqn:swing_eq} corresponds to the CSI phenomenon.  
Because of $\delta\in\bbT^1$, the unbounded motion involves the unbounded trajectory in the $\omega$-direction, that is, the average of deviation of rotor speeds.  
We use Corollaries~\ref{thm:2-IM} and \ref{thm:NEW} to analyze invariant sets of the flow (defined by the system \eqref{eqn:swing_eq-t}), in which all the generators show \emph{bounded} deviation of rotor speeds in time.  
The analysis is crucial to understanding the so-called stability region of the loop power grid, i.e. how the grid's behavior depends on initial conditions representing failures in the grid.  

Numerical simulations are performed for analysis of invariant sets.  
To do so, we need to fix (i) a function $f$, (ii) the subset of state space on which we identify invariant sets, 
and (iii) the exit time $T\sub{ex}$ to obtain an approximation of each time-average $f^\ast$.   
We use the function $f(\delta)=\sin 2\delta$ and the grid of $401\times 401$ of initial conditions $(\delta,\omega)$ on $[1,2]\times [-0.15,0.15]$. 
The averaging operation of a single function can be used for the identification of invariant sets.  
Numerical integration of the system \eqref{eqn:swing_eq} is performed with the 4th-order symplectic integrator \cite{Yoshida_PLA150} with time step $h$:  see Appendix~\ref{appsec:3} for details.  
The parameter settings are the following:
\begin{equation}
p\sub{m}=0.95, \quad b=1, \quad N\sub{G}=20, \quad
b\sub{int}=100, \quad h=\frac{2\pi}{\itO_1}\frac{1}{N}, 
\quad T\sub{ex}=\frac{2\pi}{\itO_1}\times 2000,
\end{equation}
where $N=8$ or $16$ depending on the setting of $\mathcal{J}$.  
The values of $p\sub{m}$, $b$, $N\sub{G}$, and $ b\sub{int}$ are the same as in \cite{Susuki_JNLS09}.  

\begin{figure}[t]
\centering
\includegraphics[width=\textwidth]{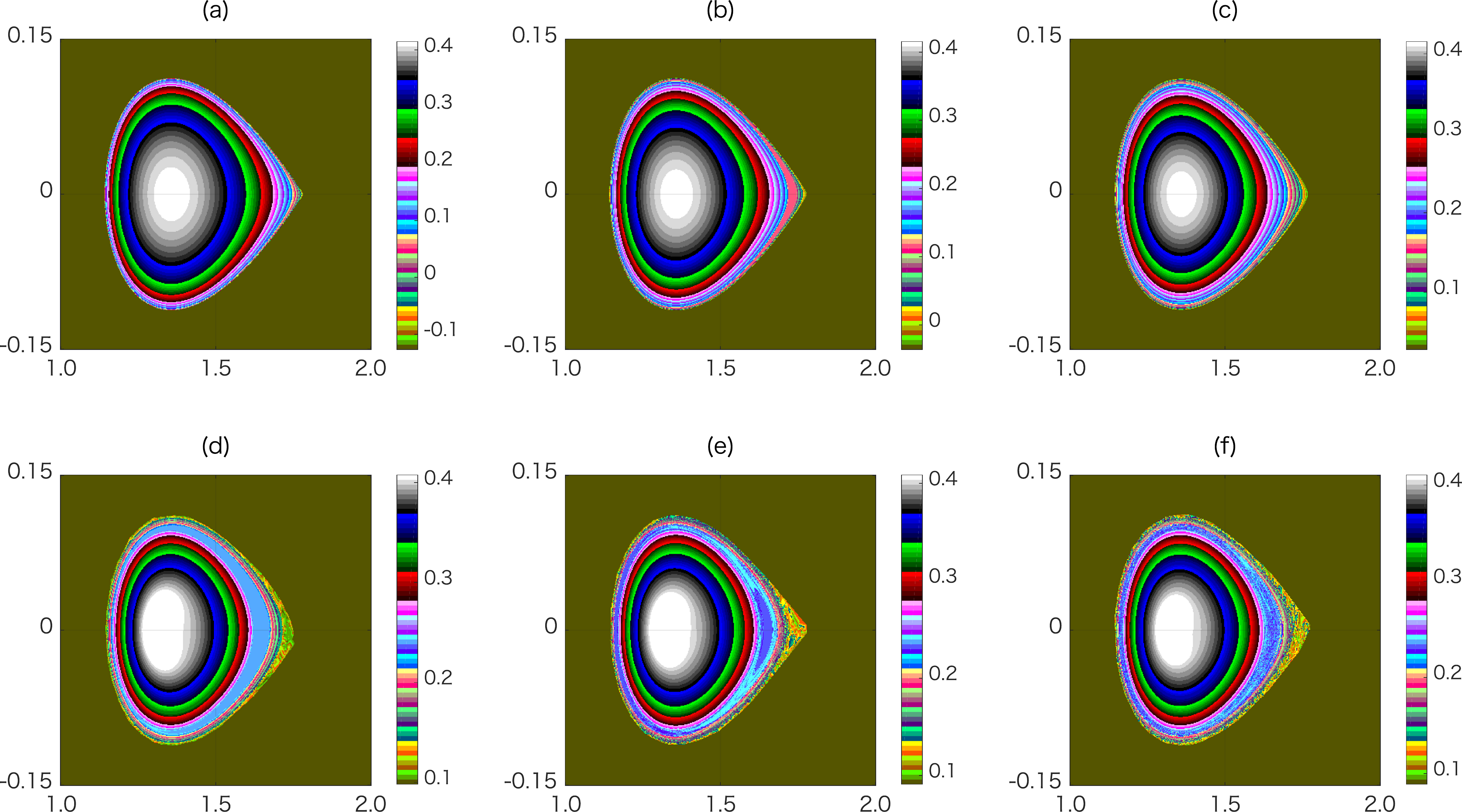}
\centering
\caption{%
Analysis of invariant sets of the measure-preserving flow defined by \eqref{eqn:swing_eq}--I: Initial phase $\vct{\theta}_0=\vct{0}$ and multiple setting of $\mathcal{J}$ i.e. (a) $\mathcal{J}=\{1\}$, (b) $\mathcal{J}=\{1,2\}$, (c) $\mathcal{J}=\{1,2,3\}$, (d) $\mathcal{J}=\{1,2,3,4\}$, (e) $\mathcal{J}=\{1,2,3,4,5\}$, and (f) $\mathcal{J}=\{1,2,3,4,5,6\}$.  
The horizontal axis for each figure is $\delta\in[1,2]$, and the vertical axis is $\omega\in[-0.15,0.15]$.  
For the outer dark-green regions, every trajectory starting from them is unbounded in time.  
}%
\label{fig:1}
\end{figure}

\begin{figure}[t]
\centering
\includegraphics[width=\textwidth]{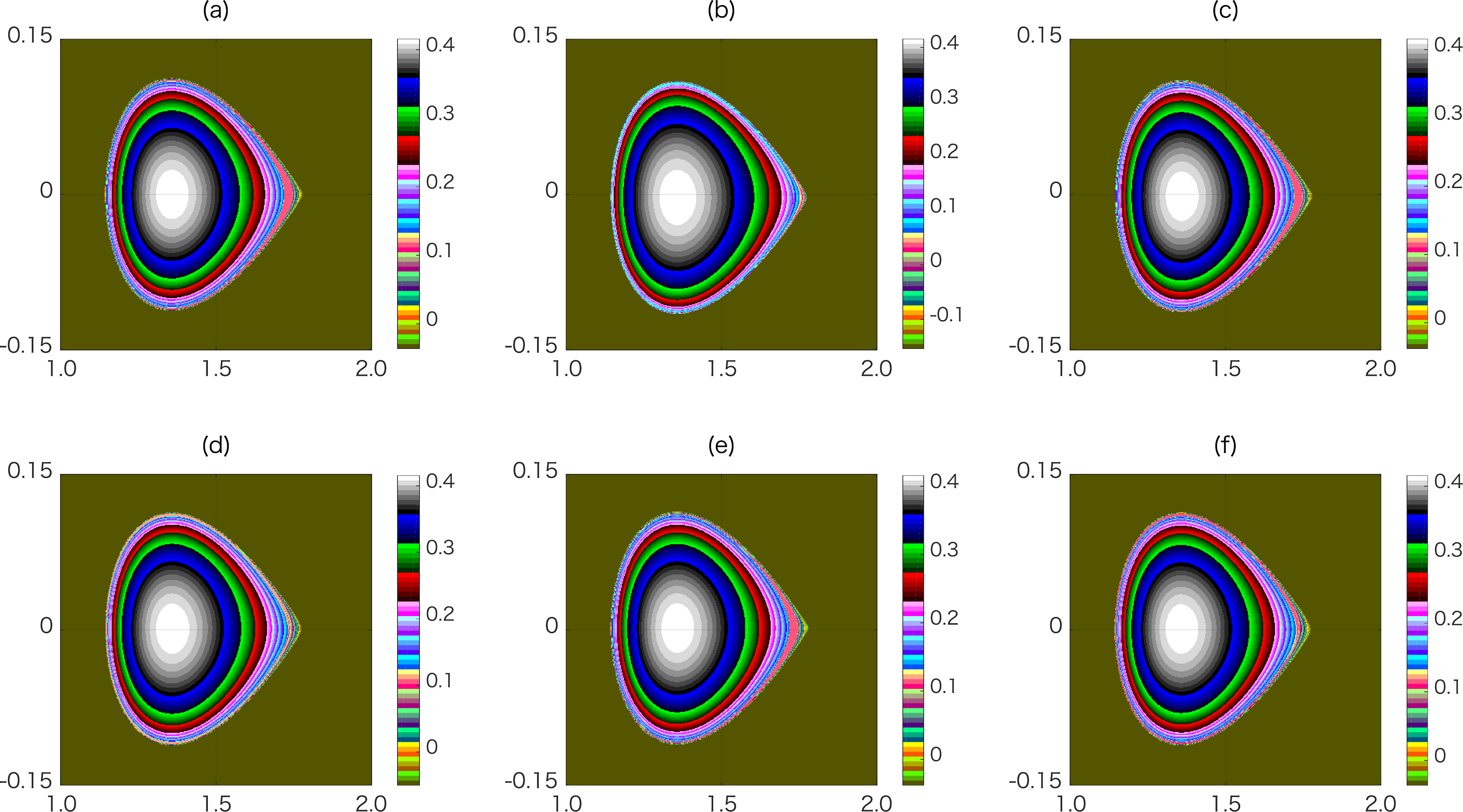}
\caption{%
Analysis of invariant sets of the measure-preserving flow defined by \eqref{eqn:swing_eq}--II: $\mathcal{J}=\{1,2\}$ and multiple setting of initial phase $\vct{\theta}_0=(\theta_{10},\theta_{20})^\top=(2\pi k\mathit{\Omega}_1/\mathit{\Omega}_2,0)^\top$ for (a) $k=0$ (same as Figure~\ref{fig:1}(b)), (b) $k=1$, (c) $k=2$, (d) $k=3$, (e) $k=4$, and (f) $k=5$.  
The horizontal axis for each figure is $\delta\in[0,\pi]$, and the vertical axis is $\omega\in[-0.15,0.15]$.  
For the outer dark-green regions, every trajectory starting from them is unbounded in time.  
}%
\label{fig:2}
\end{figure}

Figure~\ref{fig:1} shows numerical results on analysis of invariant sets of the flow defined by the quasiperiodically forced system \eqref{eqn:swing_eq}.   
In this figure we change the number of excitation modes, i.e. $\mathcal{J}$; (a) $\mathcal{J}=\{1\}$, (b) $\mathcal{J}=\{1,2\}$, (c) $\mathcal{J}=\{1,2,3\}$, (d) $\mathcal{J}=\{1,2,3,4\}$, (e) $\mathcal{J}=\{1,2,3,4,5\}$, and (f) $\mathcal{J}=\{1,2,3,4,5,6\}$.  
The amplitude $c_j$ in the figure is common and satisfies $\displaystyle \sqrt{\sum_{j\in\mathcal{J}}c_j^2}=1.5$, implying that the root-means-square of the forcing term does not change for any setting of $\mathcal{J}$. 
All the initial phases $\vct{\theta}_0$ are set to zero.
For the outer \emph{dark-green} region in each figure, every trajectory starting from it is unbounded in time.  
Except for the \emph{dark-green} on the bottom, the color bar attached to each figure denotes the value of time-average $f^\ast(\delta)$.  
The level sets of $f^\ast(\delta)$ are colored by the \emph{same} color.   
That is, the set of the \emph{same} color belongs to one invariant set.  
By Corollary~\ref{thm:NEW}, the fact that the level set is bounded in these figures implies that the associated subset of invariant set is bounded in the augmented state space $M\times\bbT^{|\mathcal{J}|}$.  
In the figures (a,b,c), we see that the color plot of the level sets forms concentric rings.  
However, in the figures (d,e,f), we see that the color plot does does not necessarily exhibit concentric rings and does become \emph{scattered} in the bands close to the outer \emph{dark-green} regions.  
This implies that the structure of invariant sets is complicated in the bands.  
We anticipate this results from the so-called resonance phenomenon (see, e.g., \cite{Greenspan_SIAMJMA15,Ueda_IJBC8}) as an interaction between a family of bounded oscillations in the unforced system and quasiperiodic forcing.

Figure~\ref{fig:2} shows other numerical results on analysis of invariant sets.  
In this figure, we consider the forcing term with two frequencies, $\mathcal{J}=\{1,2\}$, and we change the initial phases $\vct{\theta}_0=(\theta_{10},\theta_{20})^\top=(2\pi k\mathit{\Omega}_1/\mathit{\Omega}_2,0)$ where $k=0,1,\ldots,5$.  
The color plots here are conducted in the same manner as in Figure~\ref{fig:1}, and for the outer \emph{dark-green} regions every trajectory starting from them is unbounded in time.  
By Corollary~\ref{thm:NEW}, the fact that the level set with same color is bounded in these figures implies that the associated subset of invariant set is bounded in the augmented state space.  
Here, under the current setting of parameters, it is conjectured that outside the outer \emph{dark-green} regions (i.e., outside the computational domain of the analysis), there exists no state from which trajectory is bounded in  time. 
This is true in the unperturbed case because there exists one homoclinic orbit separating the bounded and unbounded trajectories inside the computational domain. 
Thus, it can be inferred that the level sets discussed above are bounded in $M=\mathbb{T}^1\times\mathbb{R}$, implying by Corollary~\ref{thm:2-IM} that the whole of the corresponding invariant sets are \emph{uniformly} bounded in the augmented state space. 
The intersection of uniformly bounded invariant sets for all $\vct{\theta}$ corresponds to the stability region of the loop power grid, in which all the generators show bounded deviation of rotor speeds in time.

\section{Conclusions}
\label{sec:outro}

In this paper, we studied the ergodic partition and invariant sets of the quasiperiodically forced dynamical system \eqref{eqn:org_syst}.  
The main theoretical contributions of this paper are twofold.  
One is to provide a theory of ergodic partition of state space for smooth flows.  
The theory is a natural extension of that in \cite{Mezic_CHAOS9} and is applicable to measure-preserving and dissipative flows arising in various physical and engineering systems.  
Examples of them include dynamical systems induced by time-dependent Hamiltonians and incompressible fluid flows with time-dependent velocity profiles.  
The other is to provide a new characterization of invariant sets in the the quasiperiodically forced system \eqref{eqn:org_syst}, in which we introduced a concept of uniformly bounded invariant sets.  
The developed theory was applied to characterize the CSI phenomenon of a rudimentary power grid.  
We have speculated that the phenomenon can be characterized, in particular, the stability region corresponds to the intersection of uniformly bounded sets for all initial phases; or a sufficient condition for the phenomenon is that the operating state of the grid is placed outside of the bounded sets at a particular initial phase or time like $t=0$.

\section*{Acknowledgments}

Y.S. thanks Dr.~Marko $\rm Budi\check{s}i\acute{c}$ for his introduction to theory and computation of ergodic partition and fruitful discussions.  
The authors also appreciate the reviewers for their valuable suggestion of the manuscript. 
During part of the work on this paper, Y.S. was at the Department of Mechanical Engineering, University of California, Santa Barbara, United States, and at the Department of Electrical Engineering, Kyoto University, Japan.

\appendix

\section{Proof of Lemma~\ref{lemma:1}}
\label{appsec:1}

A continuous function $f$ on $X$ is measurable and, from the assumption that $X$ is compact for the theoretical analysis in Section~\ref{sec:ergodic}, $f$ is bounded on $X$.  
Here, we note that the time-average $f^\ast$ of the measurable function $f$ is measurable as a limit of measurable functions $f_T$ on $X$, defined as
\[
f_T(x):=\frac{1}{T}\int^T_0 f(\vct{S}^t(x)){d}t \qquad T>0.
\] 
Since we consider the sets $A_\alpha$ on $\itS\subset X$, the fact that the family of $A_\alpha$ is a partition of $\itS$ is obvious.  
Next, the fact that the partition, denoted by $\zeta_f$, is measurable follows by taking $\mathfrak{D}_f$ to be the collection of pre-images under $f^\ast$ of open intervals with rational endpoints in $\bbR$.  
Because $f^\ast$ is measurable, each pre-image $(f^\ast)^{-1}([a,b])$, where $a$ and $b$ are rational numbers, is measurable.  
Every set of this type is clearly separated into sets of the form $(f^\ast)^{-1}(\{c\})$, $c\in\bbR$. 
This implies that every element of $\mathfrak{D}_f$ is a union of elements of $\zeta_f$.  
Furthermore, because the set of all rational numbers is dense in $\bbR$, for any pair $\alpha,\beta\in\bbR$ satisfying $\alpha<\beta$, there exist two rational numbers $\underline{a},\overline{a}$ such that $\alpha<\underline{a}<\beta<\overline{a}$.   
The pre-image $(f^\ast)^{-1}([\underline{a},\overline{a}])$ is an element of $\mathfrak{D}_f$ which we denote by $D$.  
Obviously, we see $A_\alpha\subset {D}^{\rm c}$ and $A_\beta\subset {D}$.  
Thus, it follows that $\mathfrak{D}_f$ is a basis for $\zeta_f$, and we conclude that $\zeta_f$ is measurable.

\section{Proof of Theorem~\ref{thm:1}}
\label{appsec:2}

Let $A$ be an element of $\zeta\sub{e}$.  
For a.e. point $x\in A$, the time-average $f^\ast(x)$ exists for all $f\in \mathcal{C}(X)$. 
Thus, the following linear functional $L_{A}$ on $\mathcal{C}(X)$ is well-defined:
\begin{equation}
L_{A}(f):=\lim_{T\rightarrow\infty}\frac{1}{T}\int^T_0 f(\vct{S}^t (x)){d} t \qquad x\in A.
\label{eqn:L_{A}}
\end{equation}
Then, because $L_{A}$ is a positive linear functional and $L_{A}(1)=1$, by Riesz's Representation Theorem (I.8.4 in \cite{Mane:1987}) there exists a unique probability measure $\mu_{A}$ on $X$ such that 
\begin{equation}
\int_{X} f{d}\mu_A=L_{A}(f), 
\label{eqn:Riesz}
\end{equation}
for all $f\in \mathcal{C}(X)$.  
Note that $\mu_{A}$ is invariant for $\vct{S}^t$.  
To prove this, for all $t\in\bbR$ we have
\[
\int_{X} f\circ\vct{S}^t\,{d}\mu_{A}
=L_{A}(f\circ\vct{S}^t)
=L_{A}(f)
=\int_{X} f{d}\mu_{A}.
\]
The second equality is a consequence of \eqref{eqn:L_{A}}.  
For the above operation, the continuity of $\vct{S}^t$ is required.    
Because $\mathcal{C}(X)$ is dense in $\mathcal{L}^1_{\mu_A}(X)$, $\mu_{A}$ is invariant.  

Now, we prove that $\mu_{A}$ is a probability measure on $A$.  
There is a sequence of compact sets $A^{\rm c}_n$, subsets of $A^{\rm c}$, 
such that
\begin{equation}
{A}^{\rm c}_1
\subset
\cdots
\subset
{A}^{\rm c}_{n}
\subset 
{A}^{\rm c}_{n+1}
\subset
\cdots, \qquad
\mu_{A}\left(A^{\rm c}\setminus\bigcup_{n\geq 1}A^{\rm c}_n\right)=0. \qquad
\label{eqn:hoge}
\end{equation}
Here, we can show $\mu_{A}(A^{\rm c}_n)=0$ for every $A^{\rm c}_n$.  
To do this, note that by Urysohn's Lemma, for every $A^{\rm c}_n$, there is a continuous, positive function $f_n$ on $X$ that is equal (i) to one on $A^{\rm c}_n$ and (ii) to zero outside of $A^{\rm c}_{n+1}$.  
Clearly, we see $f_n=0$ on $A$.  
Therefore, because of $\int_{X}f_n{d}\mu_{A}=\int_{X\setminus A^{\rm c}_{n+1}}f_n{d}\mu_{A}+\int_{A^{\rm c}_{n+1}\setminus A^{\rm c}_n}f_n{d}\mu_{A}+\int_{A^{\rm c}_n}f_n{d}\mu_{A}$ and 
the positiveness of $f_n$, 
we have 
\[
0
\leq\mu_{A}(A^{\rm c}_n)
\leq\int_{X} f_n{d}\mu_{A}
=L_{A}(f_n)
=0.
\]
The measure of a union of the countable number of sets with measure zero is zero:
\begin{equation}
\mu_{A}\left(\bigcup_{n\geq 1}A^{\rm c}_n\right)=0.
\label{eqn:4}
\end{equation}
Therefore, by \eqref{eqn:hoge} and \eqref{eqn:4}, we have $\mu_{A}(A^{\rm c})=0$.  
It follows from $\mu_{A}(X)=1$ that $\mu_{A}$ is a probability measure on $A$.

Next, let us prove that $\mu_{A}$ is an ergodic measure on $A$.  
First, observe that the set of all restrictions of functions in $\mathcal{C}(X)$ to $A$, denoted by $\mathcal{C}(X)|_{A}$, is dense in the set of all $\mu_{A}$-integrable functions on $A$, denoted by $\mathcal{L}^1_{\mu_{A}}(A)$.  
To show this, note that $\mathcal{C}(X)$ is dense in $\mathcal{L}^1_{\mu_{A}}(X)$.  
Let $f$ be an element of $\mathcal{L}^1_{\mu_{A}}(A)$.  
Consider the extension of $f$ to $X$, $\bar{f}$, such that $\bar{f}=f$ on $A$ and $\bar{f}=0$ elsewhere.  
Then, we have $\bar{f}\in\mathcal{L}^1_{\mu_{A}}(X)$ because the following integral exists:
\[
\int_{X}\bar{f}{d}\mu_{A}=\int_{A}f{d}\mu_{A}.
\]
Here, since $\mathcal{C}(X)$ is dense in $\mathcal{L}^1_{\mu_{A}}(X)$, there is a sequence of functions in $\mathcal{C}(X)$, $\{f_n\}$, 
converging to $\bar{f}$.  
Thus, the corresponding sequence of restrictions, $\{f_n|_{A}\}$, converges to $f$.  
Therefore, we observe that $\mathcal{C}(X)|_{A}$ is dense in $\mathcal{L}^1_{\mu_{A}}(A)$.  
Now, by the same argument as \eqref{eqn:Riesz}, for all $f\in \mathcal{C}(X)|_{A}$ we have 
\begin{align}
\int_{A}f{d}\mu_{A}
&= L_A(f) \nonumber\\
&= f^\ast(x) \qquad x\in A.
\label{eqn:ergodic_on_A}
\end{align}
Since \eqref{eqn:ergodic_on_A} holds for the dense set $\mathcal{C}(X)|_A$ in $\mathcal{L}^1_{\mu_{A}}(A)$, $\vct{S}^t|_{A}$ is ergodic:  see Proposition~2.2 in Chapter~II of \cite{Mane:1987} for discrete-time systems.  
This proposition can be naturally extended to continuous-time systems.  
Hence, we complete the proof that there indeed exists an ergodic measure $\mu_{A}$ for any element 
$A$ of the partition $\zeta\sub{e}$.  

Finally, we consider \eqref{eqn:edt} and that $A$ is invariant.  
The equality \eqref{eqn:edt} is obtained with the proof of Theorem~6.4 in Chapter~II of \cite{Mane:1987}.  
The proof is obtained for discrete-time systems and is extended to continuous-time  systems.
By construction, the fact that $A$ is invariant is obvious.  
This completes the proof of Theorem~\ref{thm:1}.

\section{Symplectic Integration of Time-Dependent Hamiltonian Systems}
\label{appsec:3}

In Section~\ref{sec:appl}, it is required to numerically simulate the Hamiltonian system \eqref{eqn:time-H} with the time-dependent Hamiltonian function ${H}(\delta,\omega,t)$.  
Symplectic integrator \cite{Yoshida_PLA150} is normally formulated in the case of time-independent Hamiltonian functions.   
However, one can exploit the integrator in the case of time-dependent Hamiltonian functions by augmenting the original Hamiltonian system.  
Consider the $N$ degree-of-freedom Hamiltonian system with the Hamiltonian function ${H}(q,p,t)$:  for $i=1,2,\ldots,N$, 
\begin{equation}
\frac{dq_i}{dt}=\frac{\DD}{\DD p_i}{H}(q,p,t), \qquad
\frac{dp_i}{dt}=-\frac{\DD}{\DD q_i}{H}(q,p,t)
\label{eqn:original}
\end{equation}
where $q=(q_1,q_2,\ldots,q_N)^\top$, $p=(p_1,p_2,\ldots,p_N)^\top$, and $t\in\bbR$.  
Now, by replacing the time variable $t$ with one new variable $q_0$ and defining the other new variable $dp_0/dt:=-\DD H/\DD t$, we have the augmented Hamiltonian function $\bar{H}(q_0,p_0,q,p)$ as follows:
\[
\bar{H}(q_0,p_0,q,p):=p_0+{H}(q,p,q_0).
\]
Thus, the augmented Hamiltonian system of the time-independent Hamiltonian function $\bar{H}$ is derived as
\begin{equation}
\frac{dq_i}{dt}=\frac{\DD}{\DD p_i}\bar{H}(q_0,p_0,q,p), \qquad
\frac{dp_i}{dt}=-\frac{\DD}{\DD q_i}\bar{H}(q_0,p_0,q,p)
\label{eqn:augmented}
\end{equation}
where $i=0,1,\ldots,N$.  
The flow induced by trajectories of the augmented system \eqref{eqn:augmented} is divergence-free and conserves the value of the Hamiltonian function $\bar{H}$.  
Thus, by using the integrator for the augmented system, numerical simulations of the original system \eqref{eqn:original} are indirectly performed.  
Note that the accuracy of numerical integration of \eqref{eqn:augmented} is checked by estimating the value of $\bar{H}$.  
This idea is applicable to the case of non-periodic time-dependent Hamiltonian functions.


\end{document}